\documentclass[12pt, reqno]{amsart}
\usepackage{amsmath, amsthm, amscd, amsfonts, amssymb, graphicx, color}
\usepackage[bookmarksnumbered, colorlinks, plainpages]{hyperref}

\input{mathrsfs.sty}

\newtheorem{theorem}{Theorem}[section]
\newtheorem{lemma}[theorem]{Lemma}

\newtheorem{corollary}[theorem]{Corollary}
\theoremstyle{definition}

\newtheorem{conjecture}[theorem]{Conjecture}

\theoremstyle{remark}
\newtheorem{remark}[theorem]{Remark}
\numberwithin{equation}{section}
\begin{document}

\title[Squaring operator inequalities]{Squaring operator P\'{o}lya--Szeg\"{o} and Diaz--Metcalf type inequalities}

\author[M.S. Moslehian, X. Fu] {Mohammad Sal Moslehian$^1$ and Xiaohui Fu$^{2}$}

\address{$^1$ Department of Pure Mathematics, Center of Excellence in
Analysis on Algebraic Structures (CEAAS), Ferdowsi University of
Mashhad, P. O. Box 1159, Mashhad 91775, Iran}
\email{moslehian@um.ac.ir,
moslehian@member.ams.org}

\address{$^2$ School of Mathematics and Statistics, Hainan Normal University, Haikou, 571158, P. R. China}
\email{fxh6662@sina.com }



\begin{abstract}
We square operator P\'{o}lya--Szeg\"{o} and Diaz--Metcalf type inequalities as follows: If operator inequalities $0<m_{1}^{2} \leq A\leq M_{1}^{2}$ and $0<m_{2}^{2}\leq B\leq M_{2}^{2}$ hold for some positive real numbers $m_{1}\leq M_{1}$ and $m_{2}\leq M_{2}$, then for every unital positive linear map $\Phi$ the following inequalities hold:
\begin{eqnarray*}
 (\Phi(A)\sharp\Phi(B))^2 &\leq&\left(\frac{M_1M_2 + m_1m_2}{2\sqrt{M_1M_2m_1m_2}}\right)^4\Phi(A\sharp B)^{2}
 \end{eqnarray*}
and
\begin{eqnarray*}
\left( \frac{M_2m_2}{M_1m_1}\Phi (A) + \Phi (B) \right)^2 \leq \left( \frac{(M_1m_1(M_2^2 + m_2^2) + M_2m_2(M_1^2 + m_1^2))^2}{8\sqrt{M_2M_1m_1m_2} M_1^2m_1^2M_2m_2} \right)^2\Phi (A\sharp B)^2\,.
 \end{eqnarray*}
\end{abstract}

\maketitle
\section{Introduction}

Let ${\mathbb B}({\mathscr H})$ denote the $C^*$-algebra of all bounded linear operators on a complex Hilbert space $({\mathscr H},\langle\cdot,\cdot\rangle)$. Throughout the paper, a capital letter means an operator in ${\mathbb B}({\mathscr H})$. If $\dim \mathscr{H}=n$, then $\mathbb{B}(\mathscr{H})$ can be identified with the space $\mathbb{M}_n$ of all
$n\times n$ complex matrices. We identify a scalar with the identity operator $I$ multiplied by this scalar. An operator $A$ is called positive if $\langle Ax,x\rangle\geq0$ for all $x\in{\mathscr H }$, and we then write $A\geq0$. An operator $A$ is said to be strictly positive (denoted by $A>0$) if it is a positive invertible operator. For self-adjoint operators $A, B\in{\mathbb B}({\mathscr H})$, we say $B\geq A$ if $B-A\geq0$. For strictly positive operators, $A^2\leq k^2 B^2$ for some constant $k$ if and only if $(AB^{-1})^*(AB^{-1})\leq k^2$ and this occurs if and only if $\|AB^{-1}\|\leq k$. A linear map $\Phi:{\mathbb B}({\mathscr H})\to{\mathbb B}({\mathscr K})$ is called positive if $A \geq 0$ implies $\Phi(A)\geq 0$. If this implication holds for $>$ instead of $\geq$, we say that $\Phi$ is strictly positive. It is said to be unital if $\Phi$ preserves the identity operator. The operator norm is denoted by $\|\cdot\|$. For $A, B> 0$, the operator geometric mean $A\sharp B$ is defined by $A\sharp B = A^{\frac{1}{2}}(A^{-\frac{1}{2}}BA^{-\frac{1}{2}})^{\frac{1}{2}}A^{\frac{1}{2}}$. Using a standard limit argument, this notion can be extended for positive operators $A, B$. The geometric mean operation is monotone, in the sense that $C_1\leq D_1$ and $C_2\leq D_2 $ imply that $C_1\sharp C_2\leq D_1\sharp D_2$.

Moslehian et al. \cite[Theorem 2.1]{BZ02} gave operator P\'{o}lya--Szeg\"{o} inequality (see also \cite{LEE} for an interesting proof for matrices) and Diaz--Metcalf type inequality as follows:
\begin{theorem}\label{thm1}
Let $\Phi$ be a positive linear map. If $0<m_{1}^{2} \leq A\leq M_{1}^{2}$ and $0<m_{2}^{2}\leq B\leq M_{2}^{2}$ for some positive real numbers $m_{1}\leq M_{1}$ and $m_{2}\leq M_{2}$,
 \begin{eqnarray}\label{amm}
\Phi(A)\sharp\Phi(B) \leq \alpha\cdot\Phi(A\sharp B)\,,
\end{eqnarray}
where
 \[
\alpha := \frac{1}{2}\left\{\sqrt{\frac{M_1M_2}{m_1m_2}} + \sqrt{\frac{m_1m_2}{M_1M_2}}\right\}.
\]
\end{theorem}

\begin{theorem}\label{thm2}
Let $\Phi$ be a positive linear map. If $0<m_{1}^{2} \leq A\leq M_{1}^{2}$ and $0<m_{2}^{2}\leq B\leq M_{2}^{2}$ for some positive real numbers $m_{1}\leq M_{1}$ and $0<m_{2}\leq M_{2}$, then the following inequality holds:
\begin{eqnarray}\label{ak}
 \frac{M_{2}m_{2}}{M_{1}m_{1}}\Phi(A)+\Phi(B) &\leq&\left(\frac{M_2}{m_1}+\frac{m_2}{M_1}\right)\Phi(A\sharp B).
 \end{eqnarray}
\end{theorem}

It is well known that $t^s\,\,(0\leq s\leq 1)$ is an operator monotone function and not so is $t^2$; see \cite{abc}. However, Fujii et al. \cite[Theorem 6]{FUJ} use the Kantrovich inequality to show that $t^2$ is order preserving in a certain sense as follows:
\begin{theorem}\label{fuj}
Let $0 < m \leq A \leq M$ and $A \leq B$. Then
\[
A^2 \leq \frac{(M + m)^2}{4Mm}B^2.
\]
\end{theorem}
Other similar results for the power $p$ instead of $2$ was given by Furuta in \cite{FUR}. Lin \cite{LIN1} nicely reduced the study of squared operator inequalities to that of some norm inequalities, see also \cite{LIN2}. This paper intends to square the operator P\'{o}lya--Szeg\"{o} inequality (\ref{amm}) and the Diaz--Metcalf type inequality (\ref{ak}) in different ways based on the above considerations.
\section{Results}

We start our work with the following known result.
\begin{lemma}\label{lem3}\cite{AND} Let $\Phi$ be any positive linear map and $A, B \geq 0$. Then
\begin{eqnarray}\label{e3}
\Phi(A\sharp B)\leq \Phi(A)\sharp \Phi(B)\,.
\end{eqnarray} \end{lemma}
Our first main result reads as follows.

\begin{theorem}\label{thm3}
Let $\Phi$ be a unital positive linear map. If $0<m_{1}^{2} \leq A\leq M_{1}^{2}$ and $0<m_{2}^{2}\leq B\leq M_{2}^{2}$ for some positive real numbers $m_{1}\leq M_{1}$ and $m_{2}\leq M_{2}$, then the following inequality holds:
 \begin{eqnarray}\label{af}
\left\{\Phi(A)\sharp\Phi(B)\right\}^2 \ \leq\ \beta\cdot\Phi(A\sharp B)^2\,,
\end{eqnarray}
where
 \begin{eqnarray}\label{beta}
\beta := \left\{\begin{array}{ll} \alpha^4\qquad \qquad
 & \quad \mbox{if\quad $\alpha^2 \leq \sqrt{\frac{M_1M_2}{m_1m_2}},$} \\
 & \\
\sqrt{\frac{M_1M_2}{m_1m_2}}\left\{2\alpha^2 - \sqrt{\frac{M_1M_2}{m_1m_2}}\right\} &\quad \mbox{if\quad $\alpha^2 \geq \sqrt{\frac{M_1M_2}{m_1m_2}},$}
\end{array}\right.
\end{eqnarray} and $\alpha$ is the number given in Theorem \ref{thm1}.
\end{theorem}
\begin{proof} Put
\[
M := M_1M_2,\quad m := m_1m_2, \quad C := \Phi(A\sharp B),\quad D := \Phi(A)\sharp\Phi(B)
\]
and
\begin{equation}\label{(1)}
\alpha = \frac{1}{2}\left(\frac{M + m}{\sqrt{Mm}}\right) \geq 1.
\end{equation}
It follows from the monotone property of the operator geometric mean and
$$m_1^2 \leq A \leq M_1^2\quad \mbox{and} \quad m_2^2 \leq B \leq M_2^2$$
that
\begin{eqnarray}\label{order}
\sqrt{m_1^2m_2^2} \leq A\sharp B \leq \sqrt{M_1^2M_2^2}\quad \hbox{and}\quad \sqrt{m_1^2m_2^2} \leq \Phi(A)\sharp \Phi(B) \leq \sqrt{M_1^2M_2^2} \,,
\end{eqnarray}
whence
\begin{equation}\label{(2)}
0 <m \leq C \leq M \quad{\rm and}\quad 0 < m \leq D \leq M
\end{equation}
Theorem \ref{thm1} and inequality \eqref{e3} yield that
\begin{equation}\label{(3)}
C \leq D \leq \alpha C.
\end{equation}
From \eqref{(2)} we have $(M - C)(C - m) \geq 0$ and $(M - D)(D - m) \geq 0$. Hence
\begin{equation}\label{(4)}
C^2 \leq (M + m)C - Mm \quad {\rm and}\quad D^2 \leq (M + m)D - Mm.
\end{equation}
Employing \eqref{(3)} and \eqref{(4)} we have
\begin{equation}\label{(5)}
0 \leq C^{-1}D^2C^{-1} \leq \{\alpha (M + m)C - Mm\}C^{-2}.
\end{equation}
Consider the real function $f(t)$ on $(0,\infty)$ defined as
\begin{equation}\label{(6)}
f(t) := \frac{\alpha(M + m)t - Mm}{t^2}.
\end{equation}
Then we can conclude from \eqref{(2)}, \eqref{(5)} and \eqref{(6)} that
\begin{equation}\label{(7)}
C^{-1}D^2C^{-1} \leq {\rm max}_{m \leq t \leq M}f(t).
\end{equation}
Notice that
\[
f(m) = \frac{\alpha(M + m) - M}{m} \geq \frac{\alpha(M + m) - m}{M} = f(M)
\]
and
\begin{equation}\label{(8)}
f'(t) = \frac{2Mm - \alpha(M + m)t}{t^3}.
\end{equation}
The function $f(t)$ has only one stationary (= maximum) point at
\begin{equation}\label{(9)}
t_0 := \frac{2Mm}{\alpha(M + m)}
\end{equation}
with the maximum value (by \eqref{(1)})
\begin{equation}\label{(10)}
f(t_0) = \frac{\alpha^2(M + m)^2}{4Mm} = \alpha^4.
\end{equation}
Therefore we can conclude that
\[
\max_{m \leq t \leq M} f(t) \leq \left\{\begin{array}{rl}
f(t_0) \quad &\quad \mbox{if \ $m \leq t_0$} \\
f(m) \quad &\quad \mbox{if \ $m \geq t_0.$}
\end{array}\right.
\]
It is immediate to see from \eqref{(1)} and \eqref{(9)} that $m \leq t_0$ is equivalent to $\alpha^2 \leq \sqrt{\frac{M}{m}}.$ Finally we have again by \eqref{(1)}
\begin{equation}\label{(11)}
f(m) = 2\alpha^2\sqrt{\frac{M}{m}} - \frac{M}{m}
= \sqrt{\frac{M}{m}}\left\{2\alpha^2 - \sqrt{\frac{M}{m}}\right\}.
\end{equation}
This completes the proof of Theorem
\end{proof}
\begin{remark}

The axiomatic theory for operator means of positive invertible operators have been developed by Kubo and Ando \cite{ando}. There exists an affine order isomorphism between the class of operator means $\sigma$ and the class of positive operator monotone functions
$f$ defined on $(0,\infty)$ via $f(t)I=I\sigma(tI)\,\,(t>0)$ and $A\sigma
B=A^{1\over 2}f(A^{-1\over2}BA^{-1\over2})A^{1\over2}$. The function $f_{\sharp_\mu}(t)=t^\mu$ on $(0,\infty)$ for $\mu\in(0,1)$ gives the operator weighted geometric mean $A\sharp_\mu B=A^{\frac{1}{2}}\left(A^{\frac{-1}{2}}BA^{\frac{-1}{2}}\right)^{\mu}A^{\frac{1}{2}}$. The case $\mu=1/2$ gives rise to the geometric mean $A\sharp B$. It should be noted that
\begin{eqnarray}\label{e32}
\Phi(A\sigma B) \leq \Phi(A)\sigma \Phi(B)
\end{eqnarray}
is known as an Ando type inequality in the literature, see \cite{AND}. We need the next result appeared in \cite{abc} in some general forms.
\begin{theorem}\label{gg}
Let $\Phi$ be a unital positive linear map and $\sigma$ be an operator mean with the representing function $f$. If $0<m_{1}^{2} \leq A\leq M_{1}^{2}$ and $0<m_{2}^{2}\leq B\leq M_{2}^{2}$ for some positive real numbers $m_{1}\leq M_{1}$ and $m_{2}\leq M_{2}$, then
\begin{eqnarray}\label{akk}
\Phi(A)\sigma \Phi(B) \leq \alpha \Phi(A\sigma B)
\end{eqnarray}
with
\begin{equation}\label{(17)}
\alpha=\max\left\{\frac{f(t)}{\mu_f t+\nu_f}: \frac{m_2^2}{M_1^2}\leq t\leq \frac{M_2^2}{m_1^2}\right\},
\end{equation}
where $$\mu_f=\frac{f(M_2^2/m_1^2)-f(m_2^2/M_1^2)}{(M_2^2/m_1^2)-(m_2^2/M_1^2)}\quad \hbox{and} \quad \nu_f=\frac{(M_2^2/m_1^2)f(m_2^2/M_1^2)-(m_2^2/M_1^2)f(M_2^2/m_1^2)}{(M_2^2/m_1^2)-(m_2^2/M_1^2)}\,.$$
\end{theorem}
\noindent If $\sigma=\sharp_\mu\,\,(\mu\in[0,1])$, then
{\footnotesize\begin{align*}
\alpha= \frac{\mu^\mu \left((M_2^2/m_1^2)-(m_2^2/M_1^2)\right) \left((M_2^2/m_1^2)(m_2^2/M_1^2)^\mu-(m_2^2/M_1^2)(M_2^2/m_1^2)^\mu\right)^{1-\mu}}{(1-\mu)^{\mu-1}\left((M_2^2/m_1^2)^\mu-(m_2^2/M_1^2)^\mu\right)^{\mu}},
\end{align*}}
see \cite[Theorem 3]{SEO1}. In particular, for $\sigma=\sharp $ we get the operator P\'{o}lya--Szeg\"{o} inequality.
It is easy to see that, by utilizing the same argument as in Theorem \ref{thm3}, the following general but complicated form of Theorem \ref{thm3} holds:
\end{remark}
\begin{theorem}
Under the same conditions as in Theorem \ref{thm3}
 \begin{align*}
\left(\Phi(A)\sigma \Phi(B)\right)^2 \ \leq\ \beta\cdot\Phi(A\sigma B)^2\,,
\end{align*}
where
{\footnotesize\begin{align*}
 \beta= \displaystyle{\max_{m_1^2f(m_1^{-2}m_2^2) \leq t \leq M_1^2f(M_1^{-2}M_2^2)}}\frac{\alpha(M_1^2f(M_1^{-2}M_2^2) + m_1^2f(m_1^{-2}m_2^2))t - m_1^2M_1^2f(m_1^{-2}m_2^2)f(M_1^{-2}M_2^2)}{t^2}.
\end{align*}}
\end{theorem}
Let us return to the case of usual operator geometric mean $\sharp$. An immediate consequence of Theorem \ref{thm3} reads as follows. It can be also deduced from Theorems \ref{thm1} and \ref{fuj}.
\begin{corollary}
Under the same conditions as in Theorem \ref{thm3},
\[
\left\{\Phi(A)\sharp\Phi(B)\right\}^2 \ \leq\ \alpha^4\cdot\Phi(A\sharp B)^2
\]
\end{corollary}

\begin{remark}
It is easy to see that the coefficient $\alpha^2$ from (\ref{amm}) is smaller than $\beta$ in \eqref{af}, but we obtained the relation between $( \Phi(A)\sharp\Phi(B))^2$ and $\Phi(A\sharp B)^{2}$.
\end{remark}
The following consequence may be regarded as a Gr\"uss type inequality, see \cite{MOS} and references therein.
 \begin{corollary}\label{cor1}
Let $\Phi$ be a unital positive linear map. If $0<m_{1}^{2} \leq A\leq M_{1}^{2}$ and $0<m_{2}^{2}\leq B\leq M_{2}^{2}$ for some positive real numbers $m_{1}\leq M_{1}$ and $m_{2}\leq M_{2}$, then the following inequality holds:
 \begin{eqnarray}\label{gruss}
&&(\Phi(A)\sharp\Phi(B))^2 -\Phi(A\sharp B)^{2} \leq \left(\beta-1\right)M_1^2M_2^2
\end{eqnarray}
where $\beta$ is defined by \eqref{beta}.
\end{corollary}
\begin{proof}
It follows from \eqref{af} that
 \begin{eqnarray}\label{001}
( \Phi(A)\sharp\Phi(B))^2 -\Phi(A\sharp B)^{2} &\leq& \left(\beta-1\right)\Phi(A\sharp B)^{2}.
 \end{eqnarray}
It follows from \eqref{order} that
\begin{eqnarray}\label{002}
m_1^2m_2^2 \leq \Phi(A\sharp B)^2 \leq M_1^2M_2^2\,.
\end{eqnarray}
Employing \eqref{001} and \eqref{002} we infer \eqref{gruss}.
\end{proof}
To achieve our next result we need some auxiliary lemmas. The first Lemma is a consequence of the Jensen inequality and the operator convexity of $f(t)=1/t$.
\begin{lemma}\label{lem1}\cite{CHO} Let $\Phi$ be a unital strictly positive linear map and $A>0$. Then
\begin{eqnarray}\label{e1} \Phi(A)^{-1}\leq \Phi(A^{-1})\,.
\end{eqnarray} \end{lemma}
The next lemma is proved for matrices but a careful investigation shows that it is true for operators on an arbitrary Hilbert space.
\begin{lemma}\label{lem2}\cite{BZ03} Let $A, B\ge 0$. Then
\begin{eqnarray}\label{e2}
\|AB\|\le \frac{1}{4}\|A+B\|^2\,.
\end{eqnarray} \end{lemma}
\begin{remark}
Lemma \ref{lem2} is due to Bhatia and Kittaneh in \cite[Theorem 1]{BZ03}. They proved the result for the finite dimensional case. However, for infinite dimensional case, the result for operator norm is also true . Also, we notice that if $A,B$ are compact operators, then a stronger result can be found in \cite{BZ05}.
\end{remark}
The following lemma includes the well-known operator geometric-arithmetic mean inequality.
\begin{lemma}\label{lem3}\cite[Theorem 1.27]{abc} Let $A, B \geq 0$. Then
\begin{eqnarray}
\label{e31} A\sharp B\leq \frac{A+B}{2}\,.
\end{eqnarray}
\end{lemma}

The definition of operator geometric mean easily yields the next lemma.
\begin{lemma}\label{lem4} Let $A, B>0$. Then

\begin{eqnarray}\label{e4} (A\sharp B)^{-1}=A^{-1}\sharp B^{-1}\,.
\end{eqnarray} \end{lemma}

We now present our next main result.
\begin{theorem}\label{thm4}
Let $\Phi$ be a unital positive linear map. If $0<m_{1}^{2} \leq A\leq M_{1}^{2}$ and $0<m_{2}^{2}\leq B\leq M_{2}^{2}$ for some positive real numbers $m_{1}\leq M_{1}$ and $m_{2}\leq M_{2}$, then the following inequality holds:
\begin{eqnarray}\label{aa}\begin{aligned}
 &\hspace{-1cm}\left(\frac{M_{2}m_{2}}{M_{1}m_{1}}\Phi(A)+\Phi(B)\right)^2\\ &\leq\left(\frac{(M_1m_1(M_{2}^{2}+m_{2}^{2})+M_2m_2( M_{1}^{2}+m_{1}^{2}))^2}{8\sqrt{M_2M_1m_1m_2}M_1^2m_1^2M_2m_2}\right)^2\Phi(A\sharp B)^2.\end{aligned}
 \end{eqnarray}
\end{theorem}
\begin{proof}
It is easy to see that (\ref{aa}) is equivalent to\begin{eqnarray*}\begin{aligned}
                             \left\|(\frac{M_2m_2}{M_1m_1}\Phi(A)+\Phi(B))\Phi(A\sharp B)^{-1}\right\| \leq \frac{(M_1m_1(M_{2}^{2}+m_{2}^{2})+M_2m_2( M_{1}^{2}+m_{1}^{2}))^2}{8\sqrt{M_2M_1m_1m_2}M_1^2m_1^2M_2m_2}.\end{aligned}
    \end{eqnarray*}
Evidently
    \begin{eqnarray*}
    (M_{1}^{2}-A)(m_{1}^{2}-A)A^{-1} &=& A+m_{1}^{2}M_{1}^{2}A^{-1}-m_{1}^{2}-M_{1}^{2}\leq 0,
    \end{eqnarray*}
    and hence \begin{eqnarray*}
     m_{1}^2M_{1}^2\Phi(A^{-1})+\Phi(A) &\leq & M_{1}^{2}+m_{1}^{2},
     \end{eqnarray*}
     which implies that \begin{eqnarray}\label{ff}
     M_2m_2m_{1}M_{1}\Phi(A^{-1})+\frac{M_2m_2}{M_1m_1}\Phi(A) &\leq &\frac{M_2m_2}{M_1m_1}( M_{1}^{2}+m_{1}^{2}).
     \end{eqnarray}
     In the same way, we have
\begin{eqnarray}\label{gg}
     m_{2}^2M_{2}^2\Phi(B^{-1})+\Phi(B) &\leq & M_{2}^{2}+m_{2}^{2}.
     \end{eqnarray}
Inequalities \eqref{ff} and \eqref{gg} yield that
\begin{eqnarray*}\label{ac}\begin{aligned}
 &\hspace{-2mm}\frac{M_2m_2}{M_1m_1}( M_{1}^{2}+m_{1}^{2})+M_{2}^{2}+m_{2}^{2}\\&\geq
 M_2m_2m_{1}M_{1}\Phi(A^{-1})+m_{2}^2M_{2}^2\Phi(B^{-1})+\Phi(B)+\frac{M_2m_2}{M_1m_1}\Phi(A)\\
 &\geq 2\sqrt{M_2M_1m_1m_2}M_2m_2(\Phi(A^{-1})\sharp \Phi(B^{-1}))+(\Phi(B)+\frac{M_2m_2}{M_1m_1}\Phi(A))\quad\qquad\hbox{(by \eqref{e31})}\\
 &\geq 2\sqrt{M_2M_1m_1m_2}M_2m_2\Phi(A^{-1}\sharp B^{-1})+(\Phi(B)+\frac{M_2m_2}{M_1m_1}\Phi(A))\quad\qquad\qquad\hbox{(by \eqref{e3})}\\
 &\geq 2\sqrt{M_2M_1m_1m_2}M_2m_2\Phi\left((A\sharp B)^{-1}\right)+(\Phi(B)+\frac{M_2m_2}{M_1m_1}\Phi(A))\quad\qquad\qquad\hbox{(by \eqref{e4})}\\
 &\geq 2\sqrt{M_2M_1m_1m_2}M_2m_2\Phi(A\sharp B)^{-1}+(\Phi(B)+\frac{M_2m_2}{M_1m_1}\Phi(A))\qquad\qquad\quad\quad\hbox{(by \eqref{e1})}
 \end{aligned}
\end{eqnarray*}
which implies that \begin{eqnarray}\label{acc}\begin{aligned}
                      2\sqrt{M_2M_1m_1m_2}M_2m_2\Phi(A\sharp B)^{-1}+(\Phi(B)+\frac{M_2m_2}{M_1m_1}\Phi(A))\\\leq\frac{M_2m_2}{M_1m_1}( M_{1}^{2}+m_{1}^{2})+M_{2}^{2}+m_{2}^{2}.\end{aligned}
   \end{eqnarray}
Now we have \begin{eqnarray*}\begin{aligned}
                             &\hspace{-2mm}\left\|\left(\frac{M_2m_2}{M_1m_1}\Phi(A)+\Phi(B)\right)\left(2\sqrt{M_2M_1m_1m_2}M_2m_2\Phi(A\sharp B)^{-1}\right)\right\| \\&\leq
                             \frac{1}{4}\left\|\left(\frac{M_2m_2}{M_1m_1}\Phi(A)+\Phi(B)\right)+(2\sqrt{M_2M_1m_1m_2}M_2m_2)\Phi(A\sharp B)^{-1}\right\|^2\qquad\quad\hbox{(by \eqref{e2})}\\&\leq
                             \frac{\left(M_1m_1(M_{2}^{2}+m_{2}^{2})+M_2m_2( M_{1}^{2}+m_{1}^{2})\right)^2}{4M_1^2m_1^2}\qquad\qquad\qquad\qquad\qquad\qquad\qquad\hbox{(by \eqref{acc})}
    \end{aligned}
    \end{eqnarray*}
which is equivalent to \begin{eqnarray*}\begin{aligned}
                             \left\|\left(\frac{M_2m_2}{M_1m_1}\Phi(A)+\Phi(B)\right)\Phi(A\sharp B)^{-1}\right\|&\leq \frac{(M_1m_1(M_{2}^{2}+m_{2}^{2})+M_2m_2( M_{1}^{2}+m_{1}^{2}))^2}{8\sqrt{M_2M_1m_1m_2}M_1^2m_1^2M_2m_2}.\end{aligned}
\end{eqnarray*}
\end{proof}
\begin{remark}
It is easy to obtain that
\[
\frac{M_2}{m_1}+\frac{m_2}{M_1}
\]
  in (\ref{ak}) is smaller than
\[
\frac{(M_1m_1(M_{2}^{2}+m_{2}^{2})+M_2m_2( M_{1}^{2}+m_{1}^{2}))^2}{8\sqrt{M_2M_1m_1m_2}M_1^2m_1^2M_2m_2}
 \]
in (\ref{aa}), but we obtain the relation between $\left(\frac{M_{2}m_{2}}{M_{1}m_{1}}\Phi(A)+\Phi(B)\right)^2$ and $\Phi(A\sharp B)^2$.
\end{remark}
\begin{conjecture}
Let $\Phi$ be a unital positive linear map. If $0<m_{1}^{2} \leq A\leq M_{1}^{2}$ and $0<m_{2}^{2}\leq B\leq M_{2}^{2}$ for some positive real numbers $m_{1}\leq M_{1}$ and $m_{2}\leq M_{2}$, then the following inequality hold:
\begin{eqnarray*}
(\Phi(A)\sharp\Phi(B))^2 &\leq&\frac{1}{4}\left( \sqrt{\frac{M_1M_2}{m_1m_2}}+\sqrt{\frac{m_1m_2}{M_1M_2}}\right)^2\Phi(A\sharp B)^2
 \end{eqnarray*}
and
\begin{eqnarray*}
\left( \frac{M_2m_2}{M_1m_1}\Phi (A) + \Phi (B) \right)^2 \leq \left(\frac{M_2}{m_1}+\frac{m_2}{M_1} \right)^2\Phi (A\sharp B)^2\,.
 \end{eqnarray*}

\end{conjecture}
\bigskip


\bibliographystyle{amsplain}

\begin{thebibliography}{99}


\bibitem{AND} T. Ando, \textit{Concavity of certain maps on positive definite matrices and applications to Hadamard products}, Linear Algebra Appl. \textbf{26} (1979), 203--241.

\bibitem {BZ03} R. Bhatia and F. Kittaneh, \textit{Notes on matrix arithmetic-geometric mean inequalities}, Linear Algebra Appl. \textbf{308} (2000), 203--211.

\bibitem{CHO} M.-D. Choi, \textit{A Schwarz inequality for positive linear maps on $C^*$-algebras}, Illinois J. Math. \textbf{18} (1974), 565--574.

\bibitem {BZ05} S. W. Drury, \textit{On a question of Bhatia and Kittaneh,} Linear Algebra Appl. \textbf{437} (2012), 1955-1960.

\bibitem{FUJ} M. Fujii, S. Izumino, R. Nakamoto and Y. Seo, \textit{Operator inequalities related to Cauchy-Schwarz and H\"{o}der--McCarthy inequalities}, Nihonkai Math. J. \textbf{8} (1997), no. 2, 117--122.

\bibitem{FUR} T. Furuta, \textit{Operator inequalities associated with H\"{o}der--McCarthy and Kantorovich inequalities}, J. Inequal. Appl. \textbf{2} (1998), no. 2, 137--148.

\bibitem{ando} F. Kubo and T. Ando, \textit{Means of positive linear operators}, Math. Ann. \textbf{246} (1980), 205–-224.


\bibitem{LEE} E.-Y. Lee, \textit{A matrix reverse Cauchy-Schwarz inequality}, Linear Algebra Appl. \textbf{430} (2009), 805--810.

\bibitem{LIN1} M. Lin, \textit{On an operator Kantorovich inequality for positive linear maps}, J. Math. Anal. Appl. \textbf{402} (2013), no. 1, 127--132.

\bibitem {LIN2} M. Lin, \textit{Squaring a reverse AM-GM inequality}, Studia Math. \textbf{215} (2013), 187--194.

\bibitem{MOS} J.S. Matharu and M.S. Moslehian, \textit{Gr\"uss inequality for some types of positive linear maps}, J. Operator Theory (to appear).

\bibitem {BZ02} M. S. Moslehian, R. Nakamoto and Y. Seo, \textit{A Diaz--Matcalf type inequality for positive linear maps and its applications}, Electron. J. Linear Algebra \textbf{22} (2011), 179--190.

\bibitem{abc} J. Pe\v{c}ari\'{c}, J. Mi\'{c}i\'{c} Hot and Y. Seo, \textit{Complementary inequalities to inequalities of Jensen and Ando based on Mond-Pe\v{c}ari\'c method}, Linear Algebra Appl. \textbf{318} (2000) 87--107.

\bibitem{SEO1} Y. Seo, \textit{Reverses of Ando's inequality for positive linear maps}, Math. Inequal. Appl. \textbf{14} (2011), no. 4, 905--910.
\end{thebibliography}

\end{document}